\documentclass[12pt,a4paper]{amsart}
\usepackage{amsmath}
\usepackage{stmaryrd}
\usepackage{textcomp}
\usepackage{amsmath}
\usepackage{cases}
\usepackage{amscd}
\usepackage{epsfig}
\usepackage{graphicx}

\usepackage{color}

\usepackage{verbatim}

\usepackage{amssymb}
\usepackage{amsfonts}
\usepackage{amsbsy}

\usepackage{amsthm}
\usepackage{amstext}
\usepackage{amsopn}

\newcommand{\calR}{\mathcal{R}}

\newcommand{\R}{\mathbb{R}}

\newtheorem{thm}{Theorem}[section]

\newtheorem{cor}[thm]{Corollary}

\newtheorem{prop}[thm]{Proposition}
\newtheorem{rem}[thm]{Remark}

\begin{document}
\title[Half plane geometries: zero and unbounded negative curvature]{Half plane geometries}
\author{Ioannis D. Platis \& Li-Jie Sun}

\address{Department of Mathematics and Applied Mathematics,
University of Crete,
 Heraklion Crete 70013,
 Greece.}
\email{jplatis@math.uoc.gr}
\address {Department of applied science, Yamaguchi University
2-16-1 Tokiwadai, Ube 755-8611,
Japan.}
\email{ljsun@yamaguchi-u.ac.jp}


\begin{abstract}
In this paper, we endow the right half plane with  warped product metrics. The group of holomorphic isometries of all such metrics is isomorphic to the real additive group. Of our interest are two of those metrics: they have zero and unbounded negative sectional curvature, respectively, and  both of them are not complete. 
\end{abstract}
\maketitle

\section{Warped products on the half plane}
We consider the right half plane
$
\calR=\{(r,t)\;|\;r>0,\;t\in\R\}
$
with the metric
\begin{equation}\label{eq-gr}
g_h=ds^2=dr^2+\cfrac{1}{h^2(r)}dt^2,
\end{equation}
where $h(r)$ is a positive continuous function of $r.$ This is a warped product metric of the manifold
$\R_{+}\times _{1/h(r)}\R$, see for instance \cite{P}. We compute the features of this metric explicitly.
\subsection{K\"ahler manifold}  
An orthonormal frame for $g$ comprises the vector fields
$$
\partial_r,\quad T_h=h(r)\partial_t,
$$
which satisfy the bracket relation $$[\partial_r,T_h]=(h'/h)T_h.$$ An almost complex structure $J_h$ on the tangent space is defined by the relations
$$
J_h\partial_r=T_h,\quad J_hT_h=-\partial_r.
$$
Since $\calR$ is two-dimensional, $J_h$ is a complex structure and thus $(\calR,J_h)$ is a complex manifold. The Hermitian form $\omega_h$ defined by
$
\omega_h(\partial_r,T_h)=1,
$
that is, $$\omega_h=(1/h(r))\;dr\wedge dt$$ is closed and therefore we get our first theorem.
\begin{thm}
 $(\calR,g_h,J_h,\omega_h)$ is always K\"ahler.
 \end{thm}
  This manifold  is a model for a half-plane geometry of curvature depending on the function $h$ and its derivatives. In fact, if $\nabla^h$ is the Riemannian connection, Koszul's formula
\begin{equation}\label{eq-Koszul}
g_h(\nabla^h_VU,W)=-(1/2)\left(g_h([U,W],V)+g_h([V,W],U)+g_h([U,V],W)\right),
\end{equation}
gives
\begin{eqnarray*}
&&
\nabla_{\partial_r}\partial_r=0,\quad \nabla^h_{T_h}\partial_r=-(h'/h)T_h,\\
&&
\nabla^h_{\partial_r}T_h=0,\quad \nabla_{T_h}T_h=(h'/h)\partial_r.
\end{eqnarray*}
The Riemannian curvature tensor $R^h$ is given by
\begin{equation}\label{eq-curvten}
R^h(U,V)W=\nabla^h_{V}\nabla^h_{U}W-\nabla^h_{U}\nabla^h_{V}W
+\nabla^h_{[U,V]}W
\end{equation}
It follows that $$R^h(\partial_r,T_h)\partial_r=\cfrac{h''h-2h'^2}{h^2}T_h$$ Therefore, the sectional curvature $K^h(\partial_r,T_h)$ is given by the formula
\begin{equation}\label{eq-seccurv}
K^h(\partial_r,T_h)=\cfrac{h''(r)h(r)-2(h'(r))^2}{h^2(r)}=H'(r)-H^2(r),\quad H(r)=(\log h(r))'.
\end{equation}
It follows from Equation (\ref{eq-seccurv}) that given a function $f=f(r)$, $r>0$, one may define a half-plane geometry with metric given by Equation (\ref{eq-gr}) and curvature $f(r)$, after solving the Riccati equation:
\begin{equation}\label{eq-riccati}
H'(r)-H^2(r)=f(r).
\end{equation}
In this paper we study two simple cases: a) $f\equiv 0$ and b) $f=-2/r^2$. In the first case, Equation (\ref{eq-riccati}) gives $h(r)=a_0/(a_1-r),\,a_0, a_1\in\R,$ and in the second case we have that $h(r)=c_0r/(c_1+c_2r^3),\,c_0, c_1, c_2\in\R.$ To simplify things, we choose $h_1(r)=1/r$ in the first case and $h_2(r)=r$ in the second case.   
 The following proposition now holds.
\begin{prop}
Let $ds_1^2=dr^2+r^2dt^2$ and $ds_2^2=dr^2+(1/r^2)dt^2.$ Then the holomorphic sectional curvature of the half plane with respect to the warped products $ds_1^2,$ $ds_2^2$ is $0$ and  $-2/r^2$, respectively.
\end{prop}
The flat K\"ahler case is of special interest, because of two reasons. First, due to the extraordinary nature of its geodesics and its non completeness and secondly because it can be generalised to hyperbolic spaces of all kinds and aof any dimension into a K\"ahler structure of non positive, unbounded holomorphic sectional curvature; this is the topic of discussion of another paper.  In the case of unbounded curvature here, it can be generalised to real hyperbolic spaces of any dimension. 
\section{Holomorphic isometries}\label{Sec-Hol}
It is clear that orientation preserving M\"obius transformations which stabilise $\infty$, that is, those of the form
$$
G(\zeta)=k\zeta+il,\quad\zeta=r+it,\;k>0,\;l\in\R,
$$ 
are geodesic preserving for the metric $g_h$. The group comprising all these transformations is isomorphic to the affine group ${\rm Aff}^+(\R)=\R_{+}\times\R$ and it acts on $\calR$ from the left as follows
$$
\left((a,b),\;(r,t)\right)\mapsto(ar,\;at+b).
$$
The action is transitive (all points of $\calR$ can be mapped to $(1,0)$) and free (the stabiliser of each point comprises only of $(1,0)$). 

When $h(r)=r$ or $1/r,$ such a transformation is also holomorphic with respect to the complex structure $J_h$ of $\calR$ if and only if $k=1$. To see this, write
$
G(r,t)=(kr,\;kt+l).
$
Then
$$
DG=\left(\begin{matrix}
k&0\\
0&k\end{matrix}\right),
$$
and
$$
DGJ_h=\left(\begin{matrix}
k&0\\
0&k\end{matrix}\right)\left(\begin{matrix}
0&h(r)\\
-1/h(r)&0\end{matrix}\right)=\left(\begin{matrix}
0&h(kr)\\
-1/h(kr)&0\end{matrix}\right)\left(\begin{matrix}
k&0\\
0&k\end{matrix}\right)=J_hDG
$$
if and only if $k=1.$
\begin{thm}
When $h(r)=r$ or $1/r,$ the group of holomorphic isometries ${\rm Isom}(\calR,J_h,g_h)$ comprises only mappings of the form
$$
h(\zeta)=\zeta+il, \quad \zeta=r+it,\; l\in\R,
$$
hence it is isomorphic to the additive group $\R$. 
\end{thm}
\begin{proof}
Let $G(r,t)=(u(r,t),\;v(r,t))$ be a holomorphic isometry. The Cauchy-Riemann equations are
\begin{equation}
u_r h(r)=v_t h(u),\quad u_t=-h(r)h(u)v_r.
\end{equation}
On the other hand,
\begin{eqnarray*}
G^*g_h&=&du^2+(1/h^2(u))dv^2\\
&=&(u_rdr+u_tdt)^2+(1/h^2(u))(v_rdr+v_tdt)^2\\
&=&(u_r^2+(1/h^2(u))v_r^2)\;dr^2+(u_t^2+(1/h^2(u)v_t^2)\;dt^2+2(u_ru_t+(1/h^2(u))v_rv_t)\;drdt.
\end{eqnarray*}
By the Cauchy-Riemann equations, the coefficient of the third term is
$$
2v_rv_t(1/h^2(u)-h^2(u)).
$$
 From the isometric condition $h^*g_h=g_h$, this must vanish therefore we have three possibilities: a)$v_r=0$, b) $v_t=0$ and c) $h\equiv 1$. Only a) can occur since the first contradicts the Cauchy-Riemann equations and the third tells us that the metric is Euclidean; hence $v=v(t)$.
By using once more the Cauchy-Riemann equations we obtain $u=u(r)$ and then the isometric conditions become
\begin{eqnarray*}
&&
(du/dr)^2=1,\quad (1/h^2(u))(dv/dt)^2=1/h^2(r).
\end{eqnarray*}
The first equation may only imply $du/dr=1$, hence $u=r+c$, with $c$ being a constant. Considering the second equation, one can get that $h(r)=h(u).$ When $h(r)=r$ or $1/r,$ one can easily get that $c=0$  and $v(t)=t+l$, where $l$ is a constant.
\end{proof}

\section{Geodesics} 

Let $\gamma(s)=(r(s),t(s))$ be a smooth curve defined in an interval $I$ containing zero, and suppose that $\gamma(0)=(r_0,t_0)$. The tangent vector along $\gamma$ is
$$
\dot \gamma=\dot r\partial _r+(1/h)\dot tT_h.
$$
We set $f(s)=\dot r(s)$, $g(s)=(1/h)\dot t(s)$ and we may suppose that $f^2(s)+g^2(s)=1$. Then the covariant derivative of $\dot\gamma$ is
$$
\frac{D\dot\gamma}{ds}=(\dot f+g^2H)\partial_r+(\dot g-fgH)T_h,\quad H=(\log h)'.
$$
Hence the geodesic equations are
$$
\dot f=-g^2H,\quad \dot g=fgH.
$$
It follows from $f^2+g^2=1$ that $$\dot f=(f^2-1)H.$$
If $f^2=1$ then we obtain 
$$
r(s)=\pm s+r_0,\quad t(s)=t_0,
$$
which is a straight ray passing from $(r_0,t_0)$. If $f^2-1\neq 0$, then 
$$
f^2=\dot r^2=1-b^2\exp\left(2\int H(r)dr\right), \quad b>0.
$$
This has to be positive:
$$
b\le\exp\left(-\int H(r)dr\right).
$$
Then we obtain $r(s)$ by solving the d.e.
\begin{equation}\label{eq-r}
\dot r(s)=\pm\sqrt{1-b^2\exp\left(2\int H(r)dr\right)},\quad r(0)=r_0
\end{equation}
whereas $t(s)$ is given by the solution of
\begin{equation}\label{eq-t}
\dot t(s)=\pm\frac{b\exp\left(\int H(r)dr\right)}{r(s)}, \quad t(0)=t_0.
\end{equation}
\subsection{Geodesics of $(\calR, ds_1^2)$}\label{Sec-Geo1}
In this case, $H(r)=-1/r$ and the solution to Eq. (\ref{eq-r}) is
\begin{equation}\label{equa-s}
r(s)=\sqrt{(s+a)^2+b}, 
\end{equation}
where
\begin{equation}\label{eq-r_0}
\sqrt{a^2+b}=r_0,\;b> 0.
\end{equation}
  
On the other hand, the solution to Eq. (\ref{eq-t}) is
\begin{equation}\label{equa-t}
t(s)=\pm\arctan\cfrac{s+a}{\sqrt{b}}+d,\quad d\in\R,
\end{equation}
where 
\begin{equation}\label{eq-t_0}
\pm\arctan\cfrac{a}{\sqrt{b}}+d=t_0.
\end{equation}
By setting $b=r_0^2-a^2$, $a\in(-r_0,r_0),$ we obtain the following one-parameter family of geodesics passing from $p_0$ and are not parallel to the $r$-axis:
\begin{equation}\label{eq-geo-par}
\gamma_a(s)=\left(\sqrt{s^2+2as+r_0^2},\;\pm\arctan\frac{s\sqrt{r_0^2-a^2}}{r_0^2+as}+t_0\right),\quad a\in(-r_0,r_0).
\end{equation}

\begin{thm}\label{geo-2points}
Let $p_0=(r_0,t_0)$ and $p_1=(r_1,t_1)$ be two  arbitrary distinct points  of $\calR$. The following hold:
\begin{enumerate}
\item[(i)] when $t_1=t_0,$ then $p_0$ and $p_1$ are joined by the geodesic $\gamma$ which is the horizontal line $t=t_0$;
\item[(ii)] when $t_1\neq t_0,$ then there exist geodesics $\gamma_a(s)$ with $p_0=\gamma_a(0)$ passing through $p_1=\gamma_a(s)$ if and only if $|t_1-t_0|<\pi.$ In this case
\begin{align*}
s=&\pm\sqrt{r_1^2+r_0^2\pm2r_0^2r_1^2 \cos^2(t_1-t_0)},\\
a=&\cfrac{-r_0^2\mp r_0r_1\cos(t_1-t_0)}{\pm\sqrt{r_1^2+r_0^2\pm2r_0^2r_1^2 \cos^2(t_1-t_0)}}.
\end{align*}
\end{enumerate}
\end{thm}
\begin{proof}
 It is sufficient to find the value of $s$ and $a\in(-r_0,r_0)$ such that $\gamma_a(s)=(r_1,t_1)$ when $t_1\neq t_0.$ In particular, we have the equations
\begin{equation}\label{eq-geosol1}
s^2+2as+r_0^2-r_1^2=0
\end{equation}
and 
\begin{equation}\label{eq-geosol2}
t_1-t_0=\pm\arctan\frac{s\sqrt{r_0^2-a^2}}{r_0^2+as}.
\end{equation}
We observe that $s\neq 0$; otherwise we get from (\ref{eq-geosol1}) that $r_0=r_1$ and from (\ref{eq-geosol2}) that $t_1=t_0$. 
It is now clear from (\ref{eq-geosol2}) that if $|t_1-t_0|\ge \pi$ then it has no solution. 

Let $\tau=\tan(t_1-t_0).$  It follows form (\ref{eq-geosol2}) that
$$
\tau^2=\cfrac{s^2(r_0^2-a^2)}{(r_0^2+as)^2}.
$$
Moreover, one can get that
$$(r_0+ as)^2(1+\tau^2)=r_0^2(r_0^2+2as+s^2),$$
which implies
\begin{equation}\label{geo-s2}
(r_1^2+r_0^2-s^2)^2(1+\tau^2)=4{r_0^2r_1^2},
\end{equation}
because $r_0^2+2as+s^2=r_1^2$ and ${as}=\cfrac{r_1^2-r_0^2-s^2}{2}$ from (\ref{eq-geosol1}).
\bigskip
Therefore, we have $s^2=r_1^2+r_0^2\pm 2r_0r_1\cos(t_1-t_0)$ by (\ref{geo-s2}), i.e.,
$$s=\pm\sqrt{r_1^2+r_0^2\pm2r_0^2r_1^2 \cos^2(t_1-t_0)}.$$
Considering $a=\cfrac{r_1^2-r_0^2-s^2}{2s},$ one can get that
$$a=\cfrac{-r_0^2\mp r_0r_1\cos(t_1-t_0)}{\pm\sqrt{r_1^2+r_0^2\pm2r_0^2r_1^2 \cos^2(t_1-t_0)}}.$$
By direct calculation, one can check that $a\in (-r_0, r_0).$
\end{proof}
\begin{rem}
Let $\gamma$ be a geodesic joining $p_1=(r_1,t_1)$ and $p_2=(r_2,t_2)\, (|t_1-t_2|<\pi).$
Denote by $d$ the Riemannian distance on $\calR$ induced by $g$:
$$
d(p_1,p_2)=\ell(\gamma)=\int_I\|\dot\gamma(s)\|\;ds.
$$
If $t_1=t_2=t_0,$ then $\gamma(s)=(r_1+(r_2-r_1)s,t_0)$, $\|\dot\gamma(s)\|=|r_2-r_1|$ and
$d(p_1,p_2)=|r_2-r_1|.$ 
If $t_1\neq t_2,$ let $\gamma$ be a geodesic passing from $p_1$ and $p_2.$  Then one can always find $\alpha$ such that $\gamma$ parametrised by $t$ joining $p_1, p_2$ has the following form:
$$
\gamma(t)=\left(r_1\left|\cfrac{\cos (t_1+\alpha)}{\cos (t+\alpha)}\right|, t\right),\quad 
\dot\gamma(t)=\left(r_1\left|\cfrac{\cos (t_1+\alpha)\sin (t+\alpha)}{\cos^2 (t+\alpha)}\right|, 1\right).
$$
We see that
\begin{eqnarray*}
d(p_1,p_2)&=&\left|\int_{t_1}^{t_2}r_1\cfrac{\cos (t_1+\alpha)}{\cos^2 (t+\alpha)}  \;dr\right|\\
&=& \left|r_1\cos (t_1+\alpha) \big[\tan (t+\alpha)\big]_{t_1}^{t_2}\right|\\
&=& |r_2\sin (t_2+\alpha)- r_1\sin (t_1+\alpha)|.
\end{eqnarray*}
\end{rem}
From the above proposition, we know there is not always a geodesic joining any two given points; on the other hand, the inextendible horizontal ray $(r, 0)$ $(r\in (0, 1))$ has length 1 since $\|\frac{\partial}{\partial r}\|=1$. Therefore, by Hopf-Rinow Theorem we have:
\begin{cor}
$(\calR, ds_1^2)$ is not geodesic complete. 
\end{cor}
\subsection{Geodesics of $(\calR, ds_2^2)$} 
In this cased we have $H(r)=1/r$.  The solution to Eq. (\ref{eq-t}) is now
\begin{equation}\label{eq-t2}
r(s)=\cfrac{1}{b}\sin(\pm bs+\arcsin(br_0)),
\end{equation}
where $0<b\leq1/r_0,$ and 
\begin{equation}\label{eq-t2}
t(s)=\pm\left(\cfrac{s}{2}-\cfrac{1}{4}\sin(\pm2bs+2\arcsin(br_0))\pm t_0+\cfrac{1}{4}\sin(2\arcsin(br_0))\right).
\end{equation}

We consider two distinct points $(r_0, t_0)$ and $(r_1, t_1),$ where $r_1=r_0, \, t_1\neq t_0.$ We claim that there does not exist a geodesic joining these two points when $\pi r_0>|t_1-t_0|.$ Indeed, if there exist a geodesic joining the two points, then we have that
$r_0=r_1=(1/b)\sin(\pm bs+\arcsin(br_0))$ which gives $bs=2k\pi,\,k\in\mathbb{Z}\setminus\{0\}.$
On the other hand, $$t_1-t_0=\pm\left(\cfrac{s}{2}-\cfrac{1}{4}\sin(\pm2bs+2\arcsin(br_0))+\cfrac{1}{4}\sin(2\arcsin(br_0))\right)=\pm \frac{s}{2},$$ which means $s=\pm 2(t_1-t_0).$ It follows from $b=k\pi/(t_1-t_0)$  that 
$$\cfrac{\pi}{|t_1-t_0|}\leq\left|\cfrac{k\pi}{t_1-t_0}\right|=|b|\leq
\cfrac{1}{r_0},$$
i.e. $\pi r_0\leq|t_1-t_0|.$ Consequently, there does not necessarily exist a geodesic joining any given two points. 
\begin{cor}
$(\calR, ds_2^2)$ is not geodesic complete. 
\end{cor}

\end{document}